\documentclass{article}
\usepackage[german, english]{babel}
\usepackage{amssymb,amsmath}
\usepackage{amsmath, amsthm}
\usepackage{amsfonts,epsfig}
\usepackage{tikz-qtree}
\usepackage{hyperref}

\usepackage[margin=1.1in]{geometry}
\newtheoremstyle{theorem}
  {10pt}		  
  {10pt}  
  {\sl}  
  {\parindent}     
  {\bf}  
  {. }    
  { }    
  {}     
\theoremstyle{theorem}
\newtheorem{theorem}{Theorem}

\newtheorem{proposition}{Proposition}

\newtheorem{definition}{Definition}

\newtheoremstyle{defi}
  {10pt}		  
  {10pt}  
  {\rm}  
  {\parindent}     
  {\bf}  
  {. }    
  { }    
  {}     

\newcommand{\SEPs}{{\mathcal{E}}}

\newcommand{\1}{\mathbf{1}}
\newcommand{\naturals}{\mathbb{N}}
\newcommand{\integers}{\mathbb{Z}}

\newcommand{\complex}{\mathbb{C}}

\newcommand{\ISCs}{\mathfrak{C}}
\newcommand{\per}{\ell}

\newcommand{\I}{\mathbb{I}}
\newcommand{\card}{{\rm card}}

\newcommand{\2}{\textit{\large 2}}
\newcommand{\A}{{\rm \mathbf{A}}}
\newcommand{\Hes}{{\rm \mathbf{H}}}

\newcommand{\boldxi}{{\mbox{\boldmath$\xi$}}}
\begin{document}

\bibliographystyle{plain}

\title{ \Large{\textbf{A closed form to the general solution of linear \\ difference equations with variable coeffcients}}\vspace{0.1in}\\
}

\author{A. G. Paraskevopoulos and M. Karanasos}

\maketitle
\begin{abstract} 
The determinant of a lower Hessenberg matrix (Hessenbergian) is expressed as a sum of signed elementary products indexed by initial segments of nonnegative integers.
A closed  form alternative to the recurrence expression of Hessenbergians is thus obtained. This result further leads to a closed form of the general solution for regular order linear difference equations with variable coefficients, including equations of $N$ order and equations of ascending order. \\


\end{abstract}

\section{Introduction} \label{intro-sec} 
Higher order linear difference equations with time varying coefficients (LDEVCs) and their solutions come to focus, because of their ability to capture and model the dynamics of natural and social phenomena including abrupt and structural changes.  An explicit expression for the general solution of the second order LDEVC was presented by Popenda in \cite{Pop:Exp}. A representation of their general solution in terms of a single matrix determinant was established by Kittappa in \cite{Kit:Rep}.
Closed form solutions for homogeneous LDEVC of order $N\ge 2$ have been presented by Mallik in \cite{Ma:SHo,Ma:Ho}, who also provides, in \cite{Ma:Exp}, an explicit expression for the general solution of the non-homogeneous case. Despite their theoretical significance such solution expressions have not been utilized in scientific modelling. A closed form for the general solution of LDEVC of order greater than $1$ is a long-standing problem (see \cite{El:Int}).

It has been established,   in~\cite{Pa:IGJ}, that the infinite Gauss-Jordan algorithm under a rightmost pivot elimination strategy constructs the general solution sequence of row-finite systems. This type of infinite linear systems was utilized to represent LDEVCs of regular and irregular order. Furthermore, the class of ascending order LDEVCs has been introduced in order to extend the class of $N$th order equations to cover regular order LDEVCs. It has been shown that the application of the infinite Gaussian elimination algorithm to a LDEVC of regular order generates solutions (general homogeneous and non-homogeneous) in terms of Hessenbergians. Applying the solution formula to the first order LDEVC, the well known closed form solution (see \cite{El:Int}) is recovered. Applying the same formula to the $N$th order LDEVC, the general solution obtained in \cite{Kit:Rep} is also recovered.

In this paper, we present an alternative expression to the recurrence formula (see (\ref{determinant of LHM})) for the $n$th order Hessenbergian, in closed form (see formula (\ref{closed form})). 
Unlike in the Leibniz formula for determinants, which consists of $n!$  signed elementary products (SEPs) and in which the sum variable ranges over the symmetric group of permutations, the expression obtained here is a sum of $2^{n-1}$ (non-trivial) SEPs in which the sum variable ranges over the set of integers in the initial segment $[0,2^{n-1}-1]$. This is due to an expression of each SEP as an image of a composite $\chi^{(n)}$ of two bijections, $\varphi^{(n)}$ and $\tau^{(n)}$. 
More specifically, we take advantage of the special structure of non-trivial SEPs associated with Hessenbergians (see section \ref{sec:StructuralAnalysisOfLowerHessenbergDeterminants}) to obtain (in section \ref{sec:Representation of Non-trivials SEPs by 0s and 1s})\hspace{0.01in} a direct representation of these SEPs as $n$-dimensional arrays of $0$s  and $1$s through a bijection $f^{(n)}$. The set of such arrays can be viewed as the set consisting of binary representations of the integers in $[0,2^{n-1}-1]$, denoted by ${\mathcal{B}}_{n-1}$. The function $\varphi^{(n)}$ stands for the inverse of $f^{(n)}$, which maps each $r\in{\mathcal{B}}_{n-1}$ to the SEP $\varphi^{(n)}(r)$. The bijection $\tau^{(n)}$ maps integers from $[0,2^{n-1}-1]$ to binaries in ${\mathcal{B}}_{n-1}$ and is described in terms of elementary integer functions involving the greatest and the modulo function (see (\ref{tau bijection})). 

Instead of applying the standard  transformation of the original LDEVC into another difference equation with new coefficients, as followed in~\cite{Ma:Ho}, we use the complete lower Hessenberg form of the solution matrices associated with the ascending order LDEVC. We thus obtain a closed form to the general solution for LDEVCs (see formula (\ref{closed form of general solution})) through the closed form expression of Hessenbergians.  The general solution of the $N$th order LDEVC is included as a special case.

The general solution of the $N$th order LDEVC and its closed form leads to the development of a unified theory for time series models with varying coefficients as established in \cite{PKD:UTh}. An application of this theory is the modelling of stock volatilities during financial crises as presented in~\cite{Ka:MRV}. Another application is the parsimonious formulation of periodic ARMA models \cite{PKD:UTA}.
\section{Hessenbergian solutions of linear difference equations}
\label{sec:SolutionsOfLinear DifferenceEquationsWithVariable CoefficientsInTermsOfHessenbergians}
Let $\integers$ (resp. $\integers^*$, $\integers^+$) be the set of integers (resp. non-negative integers, positive integers) and $\complex$ be the algebraic field of complex numbers.
The linear difference equation with variable coefficients (LDEVC) is defined by the recurrence\vspace{-0.07in}
\begin{equation} \label{GUO}
a_{n,0}y_{-N}+a_{n,1} y_{1-N}
+...+a_{n,N}y_{0}+a_{n,N+1}y_{1}+...+a_{n,N+n-1} y_{n-1}+a_{n,N+n} y_{n}=g_n,  \ n\in\integers^*,
\end{equation}
where $N$ is a fixed non-negative integer and $a_{n,i}, g_n\in \complex$ are values of arbitrary (complex valued) functions.

If $a_{n,N+n}\not=0$ for all $n\in\integers^*$ and $a_{m,0}\not=0$ for some $m\in\integers^*$ in (\ref{GUO}), then the LDEVC is referred to as \emph{ascending order linear difference equation} of index $N$.
The sequence of equations in (\ref{GUO}) can be written as an infinite\vspace{-0.07in} $\integers^*\times\integers^*$ linear system
\begin{equation} \label{MILS} \A\cdot y=g,
\end{equation}
where
\begin{equation} \label{coefficient matrix A} 
\A\!=\!\left(\!\!\! \begin{array}{cccccccccccc}
 a_{0,0} \!\! &\!\!\!  a_{0,1}  \!\!   &\!\!\! ... \!\!\! &\!\!  a_{0, N-1}\!\!    &\!\!\!   a_{0, N}     \! &\!\!\!  0  \!\!  &\!\!\!  ... \!\!     & \!\!\!   0   \!\!\!\!         &\!\!  0\!\!\!\!\! &...\\
 a_{1,0} \!\! &\!\!\!  a_{1,1}  \!\!   &\!\!\! ... \!\!\!    &\!\!  a_{1, N-1}\!\!   &\!\!\!  a_{1,N} \! &\!\!\! a_{1, N+1} \!\!      &\!\!\! ... \!\!    & \!\!\!     0 \!\!\!\!          &\!\!   0 \!\!\!\!\! &... \vspace{-0.05in}\\
\vdots      \!\! &\!\!\!   \vdots            \!\!   & \!\!\! \vdots \vdots \vdots  \!\!\!   &\!\!        \vdots    \!\!  &\!\!\!    \vdots       \! &\!\!\!  \vdots  \!\!   & \!\!\! \vdots \vdots \vdots  \!\!       &\!\!\!     \vdots        \!\!\!\!     &\!\!   \vdots \!\!\!\!\! & \vdots \vdots \vdots\vspace{-0.08in}\\
a_{n-1, 0}\!\! &\!\!\! a_{n-1, 1} \!\!  &\!\!\! ... \!\!\!   &\!\!  a_{n-1, N-1}\!\! &\!\!\! a_{n-1, N} \! &\!\!\! a_{n-1, N+1}\!\! &\!\!\! ... \!\!  &\!\!\!   a_{n-1, N+n-1}         \!\!\!\!    &\!\!   0\!\!\!\!\!& ... \\
 a_{n,0}  \!\!  &\!\!\! a_{n,1}  \!\! &\!\!\!  ... \!\!\! &\!\! a_{n, N-1} \!\!     &\!\!\! a_{n, N} \!\! &\!\!\!   a_{n, N+1}  \!\! &\!\!\! ... \!\!   &\!\!\! a_{n, N+n-1}   \!\!\!\!    &\!\!   a_{n, N+n}\!\!\!\!\!& ...\vspace{-0.05in}\\
 \vdots      \!\! &\!\!\!   \vdots            \!\!   & \!\!\! \vdots \vdots \vdots  \!\!\!   &\!\!        \vdots    \!\!  &\!\!\!    \vdots       \! &\!\!\!  \vdots  \!\!   & \!\!\! \vdots \vdots \vdots  \!\!       &\!\!\!     \vdots        \!\!\!\!     &\!\!   \vdots \!\!\!\!\! & \\
\end{array}\!\!\!  \right),
\end{equation}
$y=(y_{-N},y_{1-N},...,y_{-1},y_0,y_1,...)^T\in\complex^\infty$ and $g=(g_0,g_1,g_2,...)^T\in\complex^\infty$ (``$T$" stands for transposition).

If $a_{n,N+n}=0$ for some $n\in \integers^*$ and $a_{m,N+m}\not=0$ for some $m\in\integers^*$ with $m\not=n$ in (\ref{GUO}), the row-lengths of $\A$ can vary irregularly and (\ref{GUO}) is referred to as \emph{linear difference equation of irregular order}. Otherwise it would be referred to as \emph{linear difference equation of regular order}. The general solution sequence of equations of irregular order is constructed by implementing the infinite Gauss-Jordan algorithm under a rightmost pivot elimination strategy (see~\cite{Pa:IGJ}).

If $a_{n,N+n}\not=0$ for all $n\in\integers^*$, $a_{m,m}\not=0$ for some $m\in \integers^*$ and $a_{n,i}=0$ for all $i,n\in\integers^*$ such that $0\le i<n$, then (\ref{GUO}) turns into a \emph{linear difference equation with variable coefficients of order $N$}. Letting $N=0$ and $a_{n,n}\not=0$ for all $n\in \integers^*$, then (\ref{GUO}) turns into a linear difference equation with variable coefficients of \emph{unbounded order}, 
as named by Mallik in \cite{Ma:Exp}. In the terminology used herein, equations of unbounded order can be described as equations of ascending order of index $0$. Their matrix representation is non-singular, since it is lower triangular, with non-zero entries in the main diagonal. In this context, LDEVCs of ascending order and of constant order cover all LDEVCs of regular order. The most complete form of regular order LDEVCs is the ascending order one.

The coefficient matrix $\A$ in (\ref{coefficient matrix A}) associated with a LDEVC of regular order is in lower echelon form. In this case, we solely implement the infinite Gaussian elimination. This yields a unique row equivalent matrix of $\A$, say $\Hes$, called Hermite form (HF) of $\A$ (or lower row reduced echelon form of $\A$). The first $N$ opposite-sign columns of $\Hes$ augmented at their top by $N$ distinct unit vectors, turn out to be linearly independent homogeneous solution sequences of the ascending order LDEVC. As a consequence, an ascending order LDEVC of index $N$ has, as in the case of LDEVCs of $N$-order  (see \cite{El:Int}), $N$ linearly independent homogeneous solutions that span the space of homogeneous solutions, thus forming an algebraic basis of this space. This basis will be denoted by $\boldxi=\{\boldxi_i, \ 0\le i\le N-1\}$. Formally $\boldxi$ extends the notion of the \textit{fundamental solution set} associated with the $N$-order LDEVC (see \cite{El:Int}).

For every $i$ such that $0\le i\le N-1$ the fundamental solution matrix $\mbox{\boldmath$\rm \Xi$}_n^{(i)}$ is the lower Hessenberg matrix:
\[ \mbox{\boldmath$\rm \Xi$}  _n^{(i)}=
\left(\begin{array}{ccccc}
a_{0,i} &  a_{0,N}    &\!\!\!\!  0         &  ...  & 0             \\
a_{1,i}  & a_{1,N}    &  a_{1,N+1}  &  ...  & 0             \vspace{-0.07in}\\
. &     .         &  .        & ...  & .            \vspace{-0.13in} \\
.&    .         &   .        &  ...  & .            \vspace{-0.13in} \\
.&    .         &   .        &  ...  & .             \vspace{-0.07in}\\
a_{n-1,i} &   a_{n-1,N}  &  a_{n-1, N+1} &  ...  & a_{n-1, N+n-1}       \\ 
a_{n, i} &  a_{n,N}    & a_{n, N+1}   &  ...  & a_{n, N+n-1} \end{array}\right).
\]
The set $\boldxi$ of fundamental solution sequences consists of the sequences
\[\begin{array} {clccccl}
\boldxi_0     &=(1,0,...,0, & \xi_{0,0},  &\xi_{1, 0},  &... & \xi_{n,0}, &...)^T \\
\boldxi_1     &=(0,1,...,0, & \xi_{0,1},  &\xi_{1, 1},  &... & \xi_{n, 1}, &...)^T \\
  .        & .           &    .      &   .       &... &    .      &...\vspace{-0.1in}\\
  .        & .           &    .      &   .       &... &    .      &...\vspace{-0.1in}\\
  .        & .           &    .      &   .       &... &    .      &...\\
\boldxi_{N-1} &=(0,0,...,1, & \xi_{0,N-1},  &\xi_{1, N-1},  &... & \xi_{n, N-1}, &...)^T,
\end{array} \]
where the general term $\xi_{n,i}$, referred to as \emph{fundamental solution}, is given by
\[\xi_{n,i}=(-1)^{n+1}\frac{\det(\mbox{\boldmath$\rm \Xi$}  _n^{(i)})}{a_{0,N}\cdot a_{1,N+1}...a_{n, n+N}}\ \ n\ge 0.
\]

The general term of the particular solution sequence
$P=(0,0,...,0,p_0,p_1,...,p_n,...)^T$
of the ascending order LDEVC, referred to as \emph{particular solution},  is given by 
\begin{equation}\label{n-term particular solution} p_{n}=(-1)^n\frac{\det(\mbox{\boldmath$\rm P$}_n)}{a_{0,N}\cdot a_{1,N+1}...a_{n, n+N}}, \ \ n\ge 0,\end{equation}
where
\[\mbox{\boldmath$\rm P$}_n=
\left(\begin{array}{ccccc}
g_{0}  &\!\!\!\! a_{0,N}    &\!\!\!\!  0         &\!\!\!\!\!  ... \!\!\!\!\! & 0             \\
g_{1}  &\!\!\!\! a_{1,N}    & \!\!\!\! a_{1,N+1}  &\!\!\!\!\!  ... \!\!\!\!\! & 0             \vspace{-0.05in}\\
\vdots &\!\!\!\!     \vdots         & \!\!\!\!  \vdots        &\!\!\!\!\!  \vdots\vdots\vdots \!\!\!\!\! & \vdots              \vspace{-0.05in}\\
g_{n-1} &\!\!\!\!   a_{n-1,N}  & \!\!\!\! a_{n-1, N+1} &\!\!\!\!\!  ... \!\!\!\!\! & a_{n-1, N+n-1}       \\ 
g_{n} &\!\!\!\!  a_{n,N}    & \!\!\!\! a_{n, N+1}   &\!\!\!\!\!  ... \!\!\!\!\! & a_{n, N+n-1} \end{array}\right).
\]
The general solution sequence of the ascending order LDEVC, as the sum of the homogeneous  (a linear combination of fundamental solutions) and particular solutions, is given by
\begin{equation} \label{general solution}
y=(y_{-N},\! ...,y_{-1}, p_0+\sum^{N-1}_{k=0}\xi_{0,k}y_{k-N},  p_1+\sum^{N-1}_{k=0}\xi_{1,k}y_{k-N},...)^T, 
\end{equation}
where $y_{-N},...,y_{-1}$ are arbitrary constants. The \emph{general solution matrix} associated with the ascending order LDEVC is given by the lower Hessenberg matrix
\begin{equation}\label{GSMat}
\mbox{\boldmath$\rm G$}_n=
\left(\begin{array}{ccccc}
\displaystyle g_{0}\!-\!\sum_{k=0}^{N-1}a_{0,k}\ y_{k-N} &\!\! a_{0,N}    &\!\!  0         &\!\!\!  ... \!\!\! & 0  \\
\displaystyle g_{1}\!-\!\sum_{k=0}^{N-1}a_{1,k}\ y_{k-N}   &\!\! a_{1,N}    &\!\! a_{1,N+1}  &\!\!\!  ... \!\!\! & 0 \\
. &\!\!     .         &\!\!  .        &\!\!\!  ... \!\!\! & .            \vspace{-0.13in} \\
.&\!\!    .         & \!\!  .        &\!\!\!  ... \!\!\! & .            \vspace{-0.13in} \\
.&\!\!    .         & \!\!  .        &\!\!\!  ... \!\!\! & .             \\
\displaystyle g_{n-1}\!-\!\sum_{k=0}^{N-1}a_{n-1,k}\ y_{k-N} &\!\!   a_{n-1,N}  & \!\! a_{n-1, N+1} &\!\!\!  ... \!\!\! & a_{n-1, n+N-1}       \\ 
\displaystyle g_{n}\!-\!\sum_{k=0}^{N-1}a_{n,k}\ y_{k-N}  &\!\!  a_{n,N}    & \!\! a_{n, N+1}   &\!\!\!  ... \!\!\! & a_{n, n+N-1} \end{array}\right).
\end{equation}
As a result of the multilinear property of determinants, the $n$th (or general) term  $y_n=p_n+\displaystyle\sum^{N-1}_{k=0}\xi_{n,k}y_{k-N}$ of $y$ in (\ref{general solution}), referred to as \emph{general solution}, is represented in terms of Hessenbergians as follows:
\begin{equation} \label{n-term general solution}
 y_n=(-1)^n\frac{\det(\mbox{\boldmath$\rm G$}_n)}{a_{0,N}\cdot a_{1,N+1}...a_{n, n+N}}.
\end{equation}
A major objective of the present work is to provide a closed form expression for $\det(\mbox{\boldmath$\rm G$}_n)$.
\section{Hessenbergians and non-trivial signed elementary products}
\label{sec:StructuralAnalysisOfLowerHessenbergDeterminants}
Let $S_n$ be the group of permutations on $\{1,2,...,n\}$, known as symmetric group of order $n$. The signature $sgn(\per)$ of $\per\in S_n$ is defined as $-1$ if $\per$ is odd and $+1$ if $\per$ is even.  Let  $\per\in S_n$. A \emph{signed elementary product} (SEP) of a square matrix $\A=(a_{i,j})_{1\le i,j \le n}$ over $\complex$ is an ordered pair 
$\displaystyle(\per,sgn(\per)\prod_{i=1}^{n}a_{i,\per_i})$, that is an element of $S_n\times\complex$. The second component of a SEP is its \emph{product value} in $\complex$. We infer that two SEPs $(\per,sgn(\per)\prod_{i=1}^{n}a_{i,\per_i})$ and  $(l,sgn(l)\prod_{i=1}^{n}a_{i,l_i})$ of $\A$ are equal if and only if $\per=l$.
Bearing this fact in mind, we shall use the standard notation of SEPs:
$sgn(\per)a_{1,\per_1}a_{2,\per_2}...a_{n,\per_n},\ \ \per\in S_n$.
The set of SEPs of $\A$ is in one-to-one correspondence with $S_n$. Therefore the number of all SEPs of $\A$ coincides with the number of all permutations on $n$ objects, that is $\card(S_n)=n!$. The determinant of $\A$ is built out of the SEPs of $\A$, according to the Leibniz formula:  
\begin{equation} \label{Leibniz formula} \det(\A)=\sum_{\per\in S_n} sgn(\per)\prod_{i=1}^{n}a_{i,\per_i}\end{equation}
The first factor of a SEP could be any entry, say $a_{1,\per_{1}}$, from the $n$ entries of the first row of $\A$. Taking into account that $\per$ is bijective, the factor $a_{j,\per_j}$ of a SEP could be any entry from the $n-j+1$ entries of the $j$th row of $\A$ satisfying $\per_j \not=\per_1, \per_j \not=\per_2,...,\per_j \not=\per_{j-1}$.

The factors of a SEP are the nodes of the tree connected with branches, as partly displayed below:
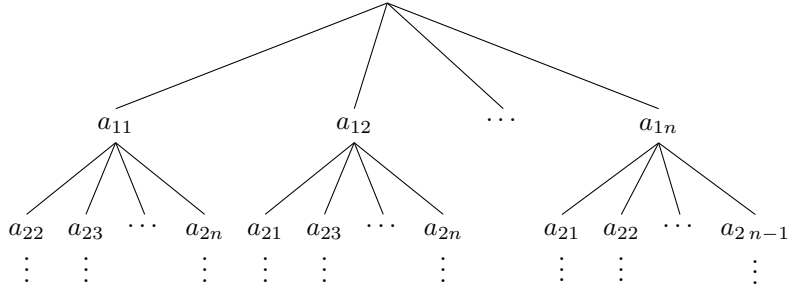
\vspace{0.1in}
 
\hspace{0.9in}\begin{tikzpicture}
\tikzset{level distance=40pt}
\tikzset{every tree node/.style={align=center,anchor=north}}
\Tree [ [.$a_{11}$ $a_{22}$\\{\vdots} $a_{23}$\vspace{-0.05in}\\{\vdots} {\ldots} $a_{2n}$\vspace{-0.05in}\\{\vdots} ] [.$a_{12}$ $a_{21}$\vspace{-0.05in}\\{\vdots} $a_{23}$\vspace{-0.05in}\\{\vdots} {\ldots} $a_{2n}$\vspace{-0.05in}\\{\vdots} ] {\ldots} [.$a_{1n}$ $a_{21}$\vspace{-0.05in}\\{\vdots} $a_{22}$\vspace{-0.05in}\\{\vdots} {\ldots} $a_{2\, n-1}$\vspace{-0.05in}\\{\vdots}  ]  ]
\end{tikzpicture}\vspace{-0.15in}\\

The $n$th order lower Hessenberg matrix over $\complex$ is an $n\times n$ matrix 
$\Hes_n=(h_{i,j})_{1\le i,j\le n}$ whose entries above the superdiagonal, called \emph{trivial}, are all zero. That is, $h_{i,j}=0$, whenever $j-i>1$, as displayed below:
\begin{equation} \label{Hessenberg Matrix}
 \Hes_n=\left(\begin{array}{cccccc}
  h_{1,1}      &  h_{1,2}   &   0        & ... &   0      & 0 \\
 h_{2,1}      &  h_{2,2}   &   h_{2,3}  & ... &   0      & 0  \vspace{-0.05in}\\
   \vdots     &  \vdots  &    \vdots        & \vdots\vdots\vdots &   \vdots           & \vdots  \vspace{-0.05in}\\
  h_{n-2, 1} & h_{n-2,2}& h_{n-2,3} & ... & h_{n-2, n-1} & 0 \\
  h_{n-1, 1} & h_{n-1,2}& c_{n-1,3} & ... & h_{n-1, n-1} & h_{n-1, n} \\
  h_{n,1}      &  h_{n,2}   & h_{n,3}     & ... & h_{n, n-1}     & h_{n, n} \end{array}\right).
\end{equation}
$\Hes_n$ can be considered as the $n$th term of the infinite chain of lower Hessenberg matrices\vspace{-0.02in}
\begin{equation}\label{chain of Hess}
\Hes_1\sqsubset\Hes_2\sqsubset...\sqsubset\Hes_n\sqsubset...
\end{equation}
where the notation $\Hes_n\sqsubset\Hes_{n+1}$ means that $\Hes_n$ is a top submatrix of $\Hes_{n+1}$.

The determinant of $\Hes_n$ for $n\ge 2$, known as \emph{Hessenbergian}, satisfies the recurrence
\begin{equation} \label{determinant of LHM}
\det(\Hes_n)=h_{n,n}\det(\Hes_{n-1})+\sum_{k=1}^{n-1}\left((-1)^{n-k}h_{n,k}\prod_{i=k}^{n-1}h_{i,i+1}\det(\Hes_{k-1})\right),
\end{equation}
where $\det(\Hes_0)=1$ and $\det(\Hes_1)=h_{1,1}$ (for a proof of the recurrence formula (\ref{determinant of LHM}) see \cite{Na:fd}).

Notice that zero entries (if any) below and including the entries of the superdiagonal are all non-trivial. A SEP of $\Hes_n$ will be called \emph{non-trivial} if it exclusively consists of non-trivial entries.  Throughout the paper the set of non-trivial SEPs  associated with $\det(\Hes_n)$ is denoted by $\SEPs_n$. 
\subsection{Hessenbergian recurrence in terms of non-trivial SEPs}
\label{sec:Recurrence Formula}
The non-trivial entries of a Hessenberg matrix $\Hes_n$ positioned below and including the main diagonal, namely $c_{i,j}=h_{i,j}$ for $j\le i$, will be called \emph{standard factors}, while the opposite-sign non-trivial entries in the superdiagonal of $\Hes_n$, namely $c_{i,i+1}=-h_{i,i+1}$, will be called \emph{non-standard factors}.
We shall also use the alternative notation to $\Hes_n$:
\begin{equation} \label{Modified Hessenberg Matrix}
 \Hes_n=\left(\begin{array}{cccccc}
c_{1,1}      & -c_{1,2}   &   0        & ... &   0           & 0  \\
c_{2,1}      & \ \  c_{2,2}  &   -c_{2,3}  & ... &   0           & 0  \\
 .          &  .        &   .        & ... &   .           & .  \vspace{-0.1in}\\
 .          &  .        &   .        & ... &   .           & .  \vspace{-0.1in}\\
  .          &  .        &   .        & ... &   .           & .  \\
c_{n-2, 1} & c_{n-2,2}& c_{n-2,3} & ... &  -c_{n-2, n-1} & 0 \\
c_{n-1, 1} & c_{n-1,2}& c_{n-1,3} & ... &\ \ \; c_{n-1, n-1} & -c_{n-1, n} \\
 c_{n,1}      &  c_{n,2}   & c_{n,3}     & ... & c_{n, n-1}     & c_{n, n}
\end{array}\right).
\end{equation}
The following Proposition will make clear the usefulness of the matrix form (\ref{Modified Hessenberg Matrix}).
\begin{proposition}\label{positively signed} 
\textit{i}\emph{)} Let $C$ be a non-trivial \emph{SEP} of $\Hes_n$ and $\textbf{c}$ be the number of non-standard factors of $C$. The expansion of \emph{(\ref{determinant of LHM})}, when written in terms of \emph{SEPs}, comprises entirely non-trivial \emph{SEPs}. Moreover any arbitrary non-trivial \emph{SEP} in this expansion, say $C=sgn(\per)h_{1,\per_1}h_{2,\per_2}...h_{n,\per_n}$, satisfies
$sgn(\per)=(-1)^\textbf{c}$, that is\vspace{-0.05in}
\begin{equation}\label{non-trivial seps} C=c_{1,\per_1}c_{2,\per_2}...c_{n,\per_n}.\end{equation}
\textit{ii}\emph{)} The number of non-trivial \emph{SEPs} of $\det(\Hes_n)$ is $2^{n-1}$.
\end{proposition}
\begin{proof}
\textit{i}) 
Writing (\ref{determinant of LHM}) in terms of entries of (\ref{Modified Hessenberg Matrix}) it takes the form:
\[\begin{array}{ll}
\det(\Hes_n)\!\!\!&=\displaystyle c_{n,n}\det(\Hes_{n-1})+\sum_{k=1}^{n-1}\left((-1)^{n-k}c_{n,k}\prod_{i=k}^{n-1}(-1)c_{i,i+1}\det(\Hes_{k-1})\right)\\
\!\!\!&=\displaystyle
c_{n,n}\det(\Hes_{n-1})\!+\sum_{k=1}^{n-1}\left(\!(-1)^{n-k}c_{n,k}(-1)^{n-k}\prod_{i=k}^{n-1}c_{i,i+1}\det(\Hes_{k-1})\!\right),
\end{array}\]
or equivalently\vspace{-0.05in}
\[\begin{array}{ll}
\det(\Hes_n)\!\!\!&=\displaystyle c_{n,n}\det(\Hes_{n-1})+\sum_{k=1}^{n-1}\prod_{i=k}^{n-1}c_{n,k}c_{i,i+1}\det(\Hes_{k-1}).
\end{array}\]
Taking into account that $\det(\Hes_0)=1$ and $\det(\Hes_1)=c_{1,1}$, the latter expression of $\det(\Hes_n)$ can be written as:\vspace{-0.07in}
\begin{equation}\label{expansion of C}
\begin{array}{ll}
\det(\Hes_n)\!=&\!\!\!\!\!c_{n,1}c_{1,2}c_{2,3}...c_{n-1, n}+c_{n,2}c_{2,3}...c_{n-1, n}c_{1,1}+ c_{n,3}c_{3,4}...c_{n-1, n}\det(\Hes_{2})\\
\!&\!\!\!\!\!\!+...+c_{n, n-1}c_{n-1, n}\det(\Hes_{n-2})+c_{n,n}\det(\Hes_{n-1})\end{array} \end{equation}
We apply induction on $n\ge 2$. Since $\det(\Hes_{2})=c_{1,1}c_{2,2}-c_{2,1}(-c_{1,2})=c_{1,1}c_{2,2}+c_{2,1}c_{1,2}$, the statement holds for $n=2$. Suppose that all the SEPs of $\Hes_{k}$ for $k\le n-1$ are non-trivial in the form: $c_{1,\per_1}c_{2,\per_2}...c_{i,\per_i}...c_{k,\per_k}$. Applying this hypothesis to the right hand side of (\ref{expansion of C}), we infer that every SEP is non-trivial in the form (\ref{non-trivial seps}), as being product of non-trivial factors. This completes the induction.\\
\textit{ii}) Let $\gamma(n)$ be the number of non-trivial SEPs of $\Hes_n$. We apply the induction on $n\ge 2$. As $\gamma(2)=2^{2-1}=2$ the statement holds for $n=2$. Suppose that the statement holds for $k\le n-1$. Then (\ref{expansion of C}) implies that: 
$\gamma(n)=1+1+\gamma(2)+\gamma(3)+...+\gamma(n-2)+\gamma(n-1)=1+1+2+2^2+...+2^{n-2}=2^{n-1}$.
This completes the induction.
\end{proof}
In view of Proposition \ref{positively signed}, the determinant in (\ref{determinant of LHM})
consists of $\card(\SEPs_n)=2^{n-1}$ non-trivial SEPs exclusively, while the formula (\ref{Leibniz formula}) yields the surplus $n!-2^{n-1}$ trivial SEPs. 
\subsection{Anatomy of non-trivial SEPs}
\label{sec:AnatomyOfSignedElementaryProductsAssociatedWithA} 
In this subsection we examine the structure of the non-trivial SEPs, as sequences of standard and non-standard factors.

\begin{proposition}\label{fundamental}
Every non-trivial factor associated with $\Hes_n$ is a factor of a non-trivial \emph{SEP} of $\Hes_n$.
\end{proposition} 
\begin{proof} 
The proof is by induction on $n\ge 2$. As $\det(\Hes_{2})=c_{1,1}c_{2,2}+c_{2,1}c_{1,2}$, the statement holds for $n=2$. Let $\Hes_{n-1}$ fulfil the statement. 
An inspection of (\ref{expansion of C}) shows that all entries of the $n$th row of $\Hes_n$ as well as the opposite-sign entries of the superdiagonal (non-standard factors) of $\Hes_n$ are factors of non-trivial SEPs of $\Hes_n$. The remaining (standard) factors of SEPs of $\Hes_n$ are entries of $\Hes_{n-1}$. The induction entails that they are factors of SEPs of $\Hes_{n-1}$. As shown in (\ref{expansion of C}), all the factors of SEPs of $\Hes_{n-1}$ are also factors of SEPs of $\Hes_{n}$ yielded by the product $c_{n,n}\det(\Hes_{n-1})$.
This completes the induction. 
\end{proof}
Propositions \ref{positively signed} and \ref{fundamental} justify the terminology ``standard and non-standard factors" adopted herein. Notice that Proposition \ref{fundamental} is not true for determinants of either lower or upper triangular matrices, in which we identify as trivial factors the zero entries in the, respectively, upper-left or lower-right corner of the matrix. Such determinants are built out of one non-trivial SEP consisting of the entries of the main diagonal exclusively.
In all that follows we adhere to the conventions:
$c_{0,0}=h_{0,0}=1$ and $\per_0=0$.
\begin{proposition}\label{elementary} Let $C=c_{1,\per_1}c_{2,\per_2}...c_{n,\per_n}$ be a non-trivial \emph{SEP} of $\Hes_n$.
If $i,k\in\integers^+$ such that $0\le k< i\le n$, then:\vspace{-0.05in}
\begin{equation}\label{Eq. elementary}
i+1>\per_{k}.
\end{equation}
\end{proposition}
\begin{proof} 
As $C$ is non-trivial, all the factors of $C$ are non-trivial. The definition of a non-trivial factor, say $c_{k,\per_{k}}$, entails that $\per_{k}-k\le 1$, whence $\per_{k}\le k+1$. The hypothesis implies that $k<i$, whence $k+1\le i$. The assertion follows from:
$\per_{k}\le k+1\le i<i+1$.
\end{proof}
Let $C=c_{1,\per_1}...c_{k,\per_k}...c_{i-1,\per_{i-1}}c_{i,\per_i}...c_{n,\per_n}$ be a non-trivial SEP of $\Hes_n$. A product 
\[C[k,i]\stackrel{{\rm def}}{=}c_{k,\per_k}c_{k+1,\per_{k+1}}...c_{i-1,\per_{i-1}}c_{i,\per_i}\] 
of consecutive factors of $C$ is said to be a \textit{string}. 
The string $C[k,i]$ is said to be a \textit{substring} of  $C[q,p]$ if both $C[k,i]$ and $C[q,p]$ are strings of the same non-trivial SEP and satisfy $q\le k$ and $p\le i$. In this case, the string $C[q,p]$ is said to be a \textit{superstring} of $C[k,i]$. If $i=k$, then $C[k,k]=c_{k,\per_k}$. 
The string $C[1,i]$ will be called \emph{initial string} determined by $i$ and it will be simply denoted by $C[i]$. The class of all initial strings determined by $i$ is denoted by $\ISCs[i]$. Evidently  $C[0]=c_{0,0}=1$, $\ISCs[0]=\{c_{0,0}\}$ and $\ISCs[1]=\{c_{1,1}, c_{1,2}\}$. Notice that we can also write $\ISCs[1]=\{c_{0,0}c_{1,1}, c_{0,0}c_{1,2}\}$. Moreover, $\ISCs[2]=\{c_{1,1}c_{2,2}, c_{1,1}c_{2,3}, c_{1,2}c_{2,1}, c_{1,2}c_{2,3}\}$. Even though the initial string $c_{1,1}c_{2,3}$ (which is included in the SEP $c_{1,1}c_{2,3}c_{3,2}$) is not a SEP, every non-trivial SEP is an initial string, since $C\in\SEPs_n$ is included in itself. In view of (\ref{chain of Hess}), a SEP of $\Hes_n$ is not a SEP of $\Hes_m$, whenever $m\not=n$. Unlike SEPs,  the string  $C[k,n]$ is also a string of $\Hes_m$ for all $m>n$.

A non-trivial factor $c_{i,j}$ is called \emph{immediate successor} (IS)  of $C[k,i-1]=c_{k,\per_k}...c_{i-1,\per_{i-1}}$, $i\ge 1$, if  $c_{k,\per_k}...c_{i-1,\per_{i-1}}c_{i,j}$ is a string. By virtue of Proposition \ref{fundamental}, every non-trivial factor $c_{i,j}$, $1\le i \le n$, associated with $\Hes_n$ is an IS of some initial string in $\ISCs[i-1]$. 
A necessary and sufficient condition for a non-trivial factor to be an IS of an initial string is given below:
\begin{proposition} \label{immediate successors 1} 
Let $i\ge 1$ and $C[i-1]=c_{1,\per_1}c_{2,\per_2}...c_{i-1,\per_{i-1}}$ be an initial string in $\ISCs[i-1]$. 
A non-trivial factor $c_{i,j}$ of $\Hes_n$ is an \emph{IS} of $C[i-1]$ if and only if $j\not=\per_1, j\not=\per_2,...,j\not=\per_{i-1}$.
\end{proposition}
\begin{proof}
If $i=1$, the factors $c_{1,1}$ and $c_{1,2}$ are the ISs of $C[0]$, since both $j=1$ and $j=2$ differ from $\per_0=0$.
If $i>1$ and $c_{i,j}$ is an IS of $C[i-1]$, then, by definition, there is $n\in\integers^+$ and a non-trivial SEP, say $B$, of the form
$B=c_{1,\per_1}c_{2,\per_2}...c_{i-1,\per_{i-1}}c_{i,\per_i}...c_{n,\per_n}$ with  $j=\per_i$. As $\per$ is bijective the result follows.

For the converse statement we assume that $j\not=\per_1, j\not=\per_2,...,j\not=\per_{i-1}$. First, we construct a non-trivial SEP of $\Hes_n$, which includes $C[i]=c_{1,\per_1}c_{2,\per_2}...c_{i-1,\per_{i-1}}c_{i,j}$.  We define: $\per_i=j,\ \per_{i+1}=i+2,...,\per_{n-1}=n$. 
If $1\le r\le n-i-1$, then $i+1\le i+r\le n-1$.
In view of (\ref{Eq. elementary}), we have: $\per_{i+r}=(i+r)+1>\per_k$ for all $k: 0\le k< i+r$. Thus,  $\per_{p}\not=\per_{q}$ if and only if $p\not=q$, provided that $1\le p\le n-1$ and $1\le q\le n-1$. Since $\{1,2,...,n\}\setminus \{\per_1,\per_2,...,\per_{n-1}\}$
is a singleton, say $\{m\}$, we define $\per_n=m$. 
Accordingly, a bijection $\per$: $1\mapsto\per_1,2\mapsto\per_2,...,i\mapsto\per_i,...,n\mapsto\per_n$ has been constructed, which determines a non-trivial SEP including $C[i]$. As $C[i]$ includes $C[i-1]$, it follows that $c_{i,j}$ is an IS of $C[i-1]$, as claimed.
\end{proof}
The above Proposition is in accord with the construction of SEPs for complete square matrices, since all the entries of these matrices are non-trivial. 
\begin{proposition}\label{immediate successors 2}\textit{i}\emph{)} Let $1\le i\le n-1$. There are two distinct \emph{ISs} of  $C[i-1]$ in $\ISCs[i-1]$. One of these \emph{ISs} is the non-standard factor $c_{i, i+1}$.\\
\textit{ii}\emph{)} Let $i=n$. There is only one \emph{IS} of $C[n-1]$, which is standard.
\end{proposition}
\begin{proof}
\textit{i})  Let $c_{i,j}$ be an IS of $C[i-1]=c_{1,\per_1}c_{2,\per_2}...c_{i-1,\per_{i-1}}$. In view of Proposition \ref{immediate successors 1} the number of all possible ISs of $C[i-1]$ coincides with the number of the non-trivial entries of the $i$th row of $\Hes_n$ satisfying $j\not=l_1, j\not=l_2,...,j\not=l_{i-1}$. 
The $i$th row contains $i+1$ non-trivial entries, and therefore there are 2 (as determined by $i+1-i-1=2$) ISs of $C[i-1]$, as asserted.  
It follows from (\ref{Eq. elementary}) that $i+1\not=\per_{i-1},...,i+1\not=\per_1$.
Proposition \ref{immediate successors 1} (applied with $j=i+1$) implies that the non-standard factor $c_{i,i+1}$ is an IS of $C[i-1]$, as claimed.

\textit{ii})  If $i=n$, then the ISs of $C[n-1]$ are entries of the $n$th row, which contains $n$ non-trivial entries. Thus there is 1 (as determined by $n-(n-1)=1$) available IS of $c_{n-1,\per_{n-1}}$, which is standard as being entry of the last row.
\end{proof}
All factors of a non-trivial SEP, say $C=c_{1,\per_1}c_{2,\per_2}...c_{n-1,\per_{n-1}}c_{n,\per_n}$, in the tree representation of $\det(\Hes_n)$ are nodes, from which two branches start, up to factor $c_{n-1,\per_{n-1}}$, from which only one branch starts. These results are partly displayed below: 
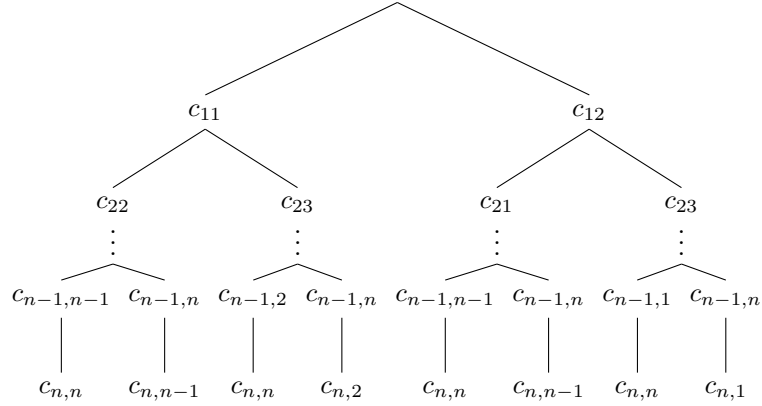
\vspace{0.04in}

\hspace{1in}\begin{tikzpicture}[sibling distance=0pt]
\tikzset{level distance=35pt}
\tikzset{level 3/.style={sibling distance=0pt}}
\tikzset{level 3/.style={level distance=40pt}}
\tikzset{level 4/.style={level distance=25pt}}
\tikzset{every tree node/.style={align=center,anchor=north}}
\Tree [ [.$c_{11}$ [.$c_{22}$\\{\vdots} [.$c_{n-1,n-1}$ $c_{n,n}$ ] [.$c_{n-1,n}$ $c_{n,n-1}$ ] ] [.$c_{23}$\\{\vdots} [.$c_{n-1,2}$ $c_{n,n}$ ] [.$c_{n-1,n}$ $c_{n,2}$ ] ] ]
[.$c_{12}$ [.$c_{21}$\\{\vdots} [.$c_{n-1,n-1}$ $c_{n,n}$ ] [.$c_{n-1,n}$ $c_{n,n-1}$ ] ] [.$c_{23}$\\{\vdots} [.$c_{n-1,1}$ $c_{n,n}$ ] [.$c_{n-1,n}$ $c_{n,1}$ ] ] ] ] ] ] 
\end{tikzpicture}\\
Furthermore, at each node is rooted one branch ending at a non-standard factor and another ending at a standard factor. All branches rooted at node  $c_{n-1,\per_{n-1}}$ end at a standard factor. The results are portrayed in the following figures: 
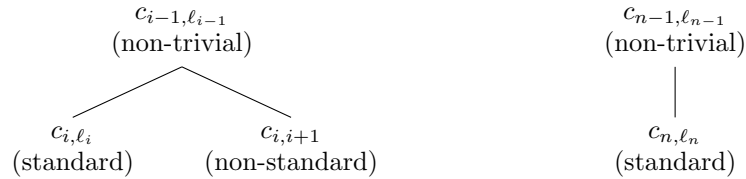
\vspace{0.04in}

\hspace{1.in}\begin{tikzpicture}[sibling distance=20pt]
\tikzset{level distance=45pt}
\tikzset{every tree node/.style={align=center,anchor=north}}
\Tree [.$c_{i-1,\per_{i-1}}$\\(non-trivial) {$c_{i,\ell_i}$\\(standard) } {$c_{i,i+1}$\\(non-standard)} ]
\end{tikzpicture}
\hspace{1in}
\begin{tikzpicture}[sibling distance=15pt]
\tikzset{level distance=45pt}
\tikzset{every tree node/.style={align=center,anchor=north}}
\Tree [.{$c_{n-1,\per_{n-1}}$\\(non-trivial)} {$c_{n,\ell_n}$\\(standard)}  ]
\end{tikzpicture}\\
The fact that the number of non-trivial SEPs of $\Hes_n$ is $2^{n-1}$ is thus re-established. Evidently $\card(\ISCs[n])=2^{n}$.
The standard ISs of strings are classified below. 
\begin{proposition}\label{standard IS} \textit{i}\emph{)} Let $1\le i \le n$. The standard \emph{IS} of any standard factor $c_{i-1, \per_{i-1}}$ is $c_{i,i}$.\\
\textit{ii}\emph{)} If $2\le i \le n$ and the factors of $C[i-1]$ are all non-standard, then the standard \emph{IS} of $C[i-1]$ is  $c_{i,1}$. \\
\textit{iii}\emph{)} Let $C[i-k-1,i-1]$ be a string such that $3\le i\le n$ and $1\le k\le i-2$. If the first factor of $C[i-k-1,i-1]$ is standard and the substring $C[i-k,i-1]$ of $C[i-k-1,i-1]$ consists exclusively of non-standard factors, that is\vspace{-0.1in}
\[\begin{array}{ccc}  \hspace{0.4in} C[i-k,i-1]= &\underbrace{c_{i-k, i-k+1}c_{i-k+1, i-k+2}...c_{i-1, i}},& \\
\hspace{0.4in} &         k  &\vspace{-0.1in}
\end{array}\]  
then the standard \emph{IS} of $C[i-k-1,i-1]$ is $c_{i,i-k}$.
\end{proposition}
\begin{proof} \textit{i}) Let us consider an arbitrary initial string $C[i-1]$, which contains the standard factor $c_{i-1, \per_{i-1}}$ (ending factor). In view of Proposition \ref{immediate successors 1}, we need to verify that none of the factors of $C[i-1]$ has as column position the index $i$. As the entries $c_{k,i}$ for $1\le k\le i-2$ are all trivial, it follows from (\ref{Eq. elementary}) that the index $i$, as being the column position of $c_{i,i}$, is not a column index of the predecessors of $c_{i-1,\per_{i-1}}$ in $C[i-1]$. As $c_{i-1,\per_{i-1}}$ is standard, we further infer that $i\not=\per_{i-1}$ and the result follows.\\
\textit{ii}) The hypothesis entails that the predecessors of $c_{i\per_i}$ are the non-standard factors: $c_{1,2},...,c_{i-1,i}$. As $\per_i=1$ is not a column index of these predecessors the assertion follows.\\
\textit{iii})  Let us consider any initial string,  $C[i-1]$, which includes $C[i-k-1,i-1]$, that is 
\[C[i-1]=c_{1\per_1}...c_{i-k-2,\per_{i-k-2}}c_{i-k-1,\per_{i-k-1}}c_{i-k, i-k+1}c_{i-k+1, i-k+2}...c_{i-1, i},\]
where $c_{1\per_1},...,c_{i-k-2,\per_{i-k-2}}$ are arbitrary.
We need to verify that none of the factors of $C[i-1]$ has as column position the index $i-k$. As $c_{i-k-1,\per_{i-k-1}}$ is standard, it follows that $\per_{i-k-1}\not=i-k$. Evidently, $i-k$ is not a column index of $c_{i-k, i-k+1},...,c_{i-2, i-1}, c_{i-1, i}$. Finally, on account of (\ref{Eq. elementary}), $i-k$ is not a column index of the predecessors of $c_{i-k-1,\per_{i-k-1}}$, since all these predecessors are trivial.
\end{proof}
The results of Proposition \ref{standard IS} can be rephrased as follows: The standard IS of a standard factor is the entry of the main diagonal in the successor row (statement 1). The standard factor whose predecessors are $k$ consecutive non-standard factors is $c_{i,i-k}$ (statement 3). As special case, if $k=i-1$, then all the predecessors of $c_{i,\per_i}$ are non-standard factors, whence $c_{i\per_i}=c_{i1}$ (statement 2).

The above results are illustrated in the following sub-tree of the tree representation of $\det(\Hes_n)$:
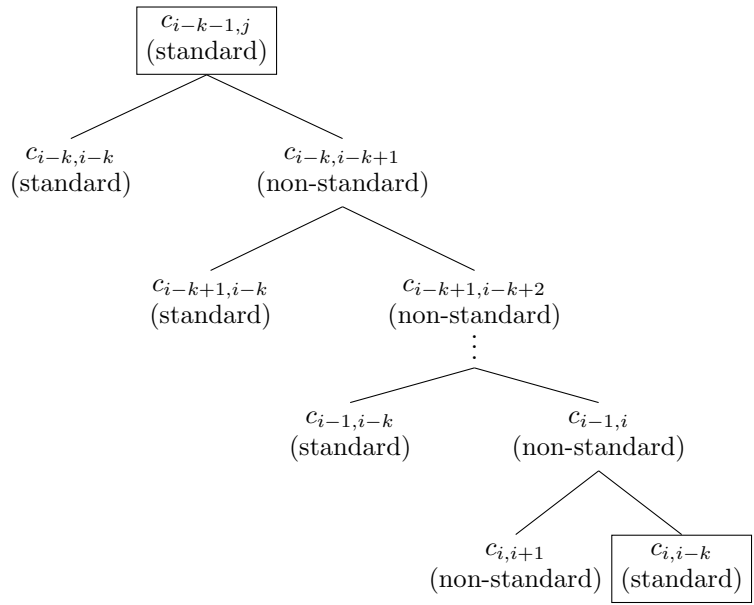
\vspace{0.2in}

\hspace{1in}\begin{tikzpicture}
[sibling distance=0pt]
\tikzset{level distance=50pt}
\tikzset{level 3/.style={sibling distance=0pt}}
\tikzset{level 3/.style={level distance=65pt}}
\tikzset{level 4/.style={sibling distance=10pt}}
\tikzset{every tree node/.style={align=center,anchor=north}}
\Tree [.\node[draw]{$c_{i-k-1,j}$\\(standard)}; $c_{i-k,i-k}$\\(standard)
[.$c_{i-k,i-k+1}$\\(non-standard)
$c_{i-k+1,i-k}$\\(standard) [.$c_{i-k+1,i-k+2}$\\(non-standard)\vspace{-0.07in}\\{\vdots} $c_{i-1,i-k}$\\(standard) 
[.$c_{i-1,i}$\\(non-standard)
$c_{i,i+1}$\\(non-standard) \node[draw]{$c_{i,i-k}$\\(standard)};  ] ] ] ] ]
\end{tikzpicture} 
\section{Non-trivial SEPs as arrays of 0s \& 1s} \label{sec:Representation of Non-trivials SEPs by 0s and 1s}
The results of the previous section are applied herein to represent each non-trivial SEP by a finite array of $0$s and 1s in one-to-one fashion. It provides in (\ref{preclosed form 1}) an alternative expression  to the recurrence (\ref{determinant of LHM}) leading closer to the desired expression in (\ref{closed form}). 
\subsection{The representation theorem}
\label{sec:AssingingArraysToSEPs}
In the rest of this paper $\2^n$ will stand for the set of functions from  $\{1,2,...,n\}$ to $\{0,1\}$, that is  $\2^n=\{(r_1,r_2,...,r_{n-1},r_n): r_i=0\ {\rm or}\ 1\}$. The set $\2^n$ can be identified with the segment of the non-negative binary integers up to and including the binary integer $2^{n}-1$ (see section \ref{sec:ElementaryIntegerFunctions}).  
The set $\mathfrak{R}_n$ is defined as the subset of $\2^n$ consisting of the elements of $\2^n$ whose last component is $r_n=1$, that is:
$\mathfrak{R}_n=\{r\in \2^n: r_n=1\}$. Evidently $\card(\2^n)=2^n$ and $\card(\mathfrak{R}_n)=2^{n-1}$.
An element $r\in\mathfrak{R}_n$ will be denoted as $r=(r_1,r_2,...,r_{n-1},1)$. In the following definition we introduce a simple rule for associating arrays in $r\in\mathfrak{R}_n$ with  non-trivial SEPs in $\SEPs_n$.

\begin{definition}\label{def:SEPs to Binaries}
We define the function $f^{(n)}:\SEPs_n\ni C=c_{1,\per_1}c_{2,\per_2}...c_{n,\per_n}\mapsto f^{(n)}(C)\in \2^n$ by:
\begin{equation}\label{def of f}
f^{(n)}(C)\stackrel{{\rm def}}{=}(r_1,r_2,...,r_{n-1},r_n): 
r_i=\left\{\begin{array}{ll} 
0, & {\rm if}\ \per_i=i+1\vspace{0.1in}\\
1,  & {\rm if}\ \per_i\not=i+1 \end{array}\right.
\end{equation}
That is, every $C=c_{1,\per_1}c_{2,\per_2}...c_{n,\per_n}\in{\mathcal{E}}_n$ is mapped through $f^{(n)}$ to $r=(r_1,r_2,...,r_{n-1},r_n)\in \2^n$, according to the rule: $r_i=0$, whenever $c_{i,\per_i}$ is non-standard or $r_i=1$, whenever $c_{i,\per_i}$ is standard.
\end{definition}
As the elements of the last row are all standard factors, the last component of $f^{(n)}(C)$ is $1$, that is: $f^{(n)}(C)\in \mathfrak{R}_n$. Therefore, $f^{(n)}:\SEPs_n\ni C\mapsto f^{(n)}(C)\in \mathfrak{R}_n$. The assumption $r_i=0$ entails that $i\not=n$, that is: $1\le i\le n-1$.
\begin{theorem}[Representation] \label{repres is bijective} The function $f^{(n)}: \SEPs_n \mapsto \mathfrak{R}_n$ in \emph{Definition} \emph{\ref{def:SEPs to Binaries}} is bijective. 
\end{theorem}
\begin{proof}
Since the set $\mathfrak{R}_n$ and the set ${\mathcal{E}}_n$ have the same number of elements ($2^{n-1}$) it suffices to show that $f^{(n)}$ is injective. Let us consider  $C=c_{1,\per_1}c_{2,\per_2}...c_{n,\per_n}$ and $P=c_{1,l_1}c_{2,l_2}...c_{n,l_n}$ in ${\mathcal{E}}_n$ such that $f^{(n)}(C)=f^{(n)}(P)$. We need to show that $C=P$ or equivalently that $\per=l$. Let us call $r=f^{(n)}(C)=f^{(n)}(P)$ and $r=(r_1,r_2,...,r_{n-1},1)$. We examine the following cases:
\begin{description}
\item[I)] If $r_i=0$, then, by Definition \ref{def:SEPs to Binaries}, the $i$th non-trivial factors both of $C$ and $P$ are non-standard, whence they coincide with the factor $c_{i,i+1}$. Thus $\per_i=l_i=i+1$.\vspace{-0.07in} 
\item[II)] If $r_i=1$, then Definition \ref{def:SEPs to Binaries} implies that the $i$th factors of $C$ and $P$ are standard. In this case, there are three possible subcases:\vspace{-0.07in}
\begin{description}
\item[IIa)] $r_i=r_{i-1}=1$ and $2\le i\le n$. \\
It follows from  Definition \ref{def:SEPs to Binaries} that the factors $c_{i-1\per_{i-1}},c_{i\per_i}$ and $c_{i-1l_{i-1}},c_{il_i}$ are all standard. By virtue of Proposition \ref{standard IS} statement (\textit{i}) the IS of any standard factor is the unique factor $c_{i,i}$. Thus $\per_i=l_i=i$.\vspace{-0.08in}
\item[IIb)] $r_i=r_{i-1-k}=1$ and $r_{m}=0$ for all $m$ such that $i-k\le m\le i-1$, whenever $3\le i\le n$ and $1\le k\le i-2$.\\
It follows from Definition \ref{def:SEPs to Binaries} that the factors $c_{i-1-k\per_{i-1-k}},c_{i\per_i}$ and $c_{i-1-kl_{i-1-k}},c_{il_i}$ are standard, while the strings $S_1=c_{i-k,\per_{i-k}},...,c_{i-1, \per_{i-1}}$ and $S_2=c_{i-k,l_{i-k}},...,c_{i-1, l_{i-1}}$  consist entirely of non-standard factors. Proposition \ref{standard IS} statement (\textit{ii}) implies that the ISs of $S_1$ and $S_2$ coincide with the factor $c_{i,i-k}$, whence $\per_i=l_i=i-k$.\vspace{-0.08in}
\item[IIc)] $r_i=1$ and $r_{m}=0$ for all $m=1,2,...,i-1$.\\
It follows from Definition \ref{def:SEPs to Binaries} that $c_{i\per_i}$ and $c_{il_i}$ are standard, while the initial strings $S_1=c_{1,\per_1},...,c_{i-1, \per_{i-1}}$ and $S_2=c_{1,l_1},...,c_{i-1, l_{i-1}}$  consist entirely of non-standard factors. Proposition \ref{standard IS} statement (\textit{iii}) implies that the ISs of $S_1$ and of $S_2$ coincide with $c_{i,1}$, whence $\per_i=l_i=1$.\vspace{-0.1in}
\end{description}
\end{description} 
The proof of the Theorem is complete.
\end{proof}
As an illustrative example, consider the non-trivial SEP:  $T=c_{1,1}c_{2,3}c_{3,2}c_{4,5}...c_{n-2, n-1}c_{n-1, 4}c_{n, n}\in\SEPs_n$ for $n\ge 8$. $T$ is represented by the array $r=(1,0,1,0,0,...,0,1,1)\in\mathfrak{R}_n$, that is $f^{(n)}(T)=r$.
Next, we verify that the inverse image of $r$ is $T$, that is $(f^{(n)})^{-1}(r)=T$. By Definition \ref{def of f} the non-standard factors occupy the same positions as the 0s in $r$, that is, the positions $i=2,4,5,...,n-2$ are occupied by the non-standard factors $c_{2,3},c_{4,5},c_{5,6},...,c_{n-2,n-1}$, respectively. Since $r_1=1$, it follows that the (unique) standard factor of $T$ is $c_{1,1}$. As $r_1=1, r_2=0, r_3=1$, Proposition \ref{standard IS}\ (\textit{ii}) entails that the third factor of $T$ is $c_{3,3-1}=c_{3,2}$. By analogy, as the number of consecutive $0$s between $r_3=1$ and $r_{n-1}=1$ is $k=n-5$, on account of $n-1-(n-5)=4$, the $(n-1)$th factor of $T$ is $c_{n-1,4}$. As $r_{n}=1$ and $r_{n-1}=1$, Proposition \ref{standard IS}\ (\textit{i}) entails that k=0, and therefore the last factor of $T$ is  $c_{n,n}$, as expected.
\subsection{An intermediate Hessenbergian expression} \label{The pre-closed form Hessenbergian}
Throughout this paper $\varphi^{(n)}$ stands for the inverse function of $f^{(n)}$, that is $\varphi^{(n)}=(f^{(n)})^{-1}$. Moreover in the determinant expansion formulas each SEP represents its
product value: $\prod_{i=1}^{n}c_{i,\per_i}$.  
Taking into account that the determinant expansion of $\Hes_n$ in (\ref{Modified Hessenberg Matrix})\vspace{0.02in} is the sum of all non-trivial SEP product values, Theorem \ref{repres is bijective} entails  that every term in the sum of $\det(\Hes_n)=\sum_{C\in {\mathcal{E}}_n}C$\vspace{0.02in} can be replaced by $\varphi^{(n)}(r)$, $r\in \mathfrak{R}_n$, that is
\begin{equation} \label{preclosed form 1}
\det(\Hes_n)=\displaystyle\sum_{r\in \mathfrak{R}_n}\varphi^{(n)}(r).
\end{equation}
The expression in (\ref{preclosed form 1}) consists entirely of $\card(\mathfrak{R}_n)=2^{n-1}$ distinct non-trivial SEPs.
The disadvantage of this formula is related to the fact that the sum variable ranges over the set of arrays in $\mathfrak{R}_n$.
\section{Hessenbergian closed form via elementary integer functions}
\label{sec:ElementaryIntegerFunctions}
In this section we introduce a suitable function which associates integers from  $\I_{n-1}=\{0,1,...,2^{n-1}-1\}$ with arrays in $\mathfrak{R}_n$ in one-to-one fashion. This will enable us to replace the indexing set $\mathfrak{R}_n$ in (\ref{preclosed form 1}) with  $\I_{n-1}$,  $n\in\integers^+$, leading to the closed form of $\det(\Hes_n)$.

Throughout the paper,  ${\mathcal{B}}_n$ denotes the set of binary integers from $0$ up to and including the number 
\[\begin{array}{lcc}\1_n=&\underbrace{11...1}& (n\ {\rm number\ of} \ 1s)\\
                 & n &\end{array}\vspace{-0.1in}\]
that is ${\mathcal{B}}_n=\{0,1,10,...,\1_n\}$. Evidently ${\mathcal{B}}_n$ consists of $2^n$ binary numbers. The binary representation of the integer $2^n-1$ is $\1_n$, that is $[2^n-1]_2=\1_n$.

Let $b\in {\mathcal{B}}_n$ with $b\not= 0$. By completing the binary figures of $b=1r_{k+1}...r_{n}$ by $k-1$ zero digits at its left up to and including the binary figure $2^{n-1}$, we adhere to the standard conventions
\[\begin{array}{clllc}1r_{k+1}...r_{n}\equiv & 0\ 0\ ...\ 0 & 1\ r_{k+1}...& r_{n} \ \ &{\rm and}\\  
&\uparrow  &\uparrow  &\uparrow&\\
{\rm Binary \ Figures:}& 2^{n-1}  & 2^{n-k}  & 2^0\ {\rm (units)}& \end{array} \ \begin{array}{cll} \textbf{0}_n\equiv &0\ \ 0 \ \ ...&0 \\
 &\uparrow  &\uparrow\\
 & 2^{n-1}  & 2^0 
 \end{array}\]
which lead to the identification of the elements of $\2^n$ with  ${\mathcal{B}}_n$. For example, if $n=5$, then the binary integer $11\in{\mathcal{B}}_5$ is identified with the array $(0,0,0,1,1)\in\2^5$.
 
Taking into account that $\card({\mathcal{B}}_{n-1})=\card(\mathfrak{R}_n)=2^{n-1}$, we define the bijection $\rho^{(n)}: {\mathcal{B}}_{n-1}\mapsto \mathfrak{R}_n$:
\begin{equation} \label{bijection rho} \begin{array}{cccccccccc}
\rho^{(n)}\!\!\!\!\!\!&(\underbrace{00...01r_{k+1}...r_{n-1}})\!\!\!\!&=\!\!\!\!&(\underbrace{0,0,...,0,1,r_{k+1},...,r_{n-1},1})&  {\rm and} &
\rho^{(n)}\!\!\!\!\!\!&(\underbrace{00...0})\!\!\!\!&=\!\!\!\!&(\underbrace{0,0,...,0,1})& \\
&  n-1  & &  n   & & &  n-1  & &  n  \end{array} \end{equation}
By identifying ${\mathcal{B}}_{n-1}$ with $\mathfrak{R}_n$ through  $\rho^{(n)}$, the function $\varphi^{(n)}$ defined above associates every binary integer $r$ in ${\mathcal{B}}_{n-1}$ with the SEP $\varphi^{(n)}(r)$.
\subsection{Nested divisions}
\label{sec:The Greatest Integer and the Modulo Functions}
Let $\kappa\in\integers^*$ and $\lambda\in\integers^+$. The largest integer not greater than the rational number $\kappa/\lambda$, will be denoted as $\lfloor \kappa/\lambda \rfloor$. 
Also $\lceil \kappa/\lambda\rceil$ denotes the smallest integer not less than $\kappa/\lambda$. The notation $\lfloor \kappa/\lambda \rfloor$ coincides with the quotient of the Euclidean division of $\kappa$ by $\lambda$, also known as \emph{integral part} (or integer part) of the number $\kappa/\lambda$.  We adopt the method of converting an integer $m\in \I_{n-1}$ into a binary number $[m]_2=r_1r_2...r_{n-1}\in{\mathcal{B}}_{n-1}$ based on the Euclidean division. The digits $r_i$ in $[m]_2$ are the remainders of nested divisions:
\[m=2q_{n-1}+r_{n-1}, q_{n-1}=2q_{n-2}+r_{n-2}, ..., q_{2}=2q_{1}+r_1.\]
Taking into account that 
\[\begin{array}{ccc}q_{n-1}=\lfloor m:2\rfloor,\ q_{n-2}=\lfloor\lfloor m:2\rfloor:2\rfloor, ..., q_1=\lfloor\lfloor...\lfloor\lfloor m:\!\!\!\!\!&\underbrace{2\rfloor:2\rfloor...\rfloor :2}&\!\!\!\!\rfloor\\
 \!\!\!\!\!&            n-1       &  \end{array} \]
the $r_i$s in $[m]_2$ can be expressed in terms of the greatest integer function ``$\lfloor \, \rfloor$" and of the modulo function:
\begin{equation} \label{eq. of rs} 
\begin{array}{ccc} r_{n-1}= m\!\!\!\!\! \mod 2,\ r_{n-2}=\lfloor m\!:\!2\rfloor\!\!\!\!\! \mod 2,..., r_1=\lfloor\lfloor...\lfloor\lfloor m\!:\!\!\!\!\!&\underbrace{2\rfloor:\!2\rfloor...\rfloor \!:\!2}&\!\!\!\!\rfloor\!\!\!\!\!  \mod 2\\
 \!\!\!\!\!&            n-2       &  \end{array}\end{equation}
We can also write $r_{n-1}=\lfloor m:2^0\rfloor\! \mod 2$, which leads to the unified expression 
\[\begin{array}{ccc} r_{i}=\lfloor\lfloor...\lfloor\lfloor m:\!\!\!\!\!&\underbrace{2\rfloor:2\rfloor...\rfloor :2}&\!\!\!\!\rfloor\!\!\!\! \mod 2,\\
 \!\!\!\!\!&            n-i-1       &
\end{array}\]
where $1\le i\le n-1$.
\\In the following Proposition we provide a condensed expression for nested divisions. 
\begin{proposition}\label{nestdiv}
Let $m\in \integers$. The following identity of nested divisions holds:
\begin{equation}\label{nested divisions} 
\begin{array}{ccc} \lfloor\lfloor...\lfloor\lfloor m:\!\!\!\!\!&\underbrace{2\rfloor:2\rfloor...\rfloor :2}&\!\!\!\!\rfloor=\lfloor m:2^k\rfloor.\\                        &   k                   & 
\end{array}\end{equation}
\end{proposition}
\begin{proof} Let $x$ be a real number and $p,q$ be positive integers.
We shall use the well known identity
\begin{equation} \label{nested division identity 1} \lfloor\lfloor x\rfloor:p\rfloor=\lfloor x:p\rfloor.
\end{equation}
Taking into account that $ (x:q):p=x:(p\cdot q)$, it follows from (\ref{nested division identity 1}) that:
\begin{equation} \label{nested division identity 2} 
\lfloor\lfloor x:q\rfloor:p\rfloor=\lfloor (x:q):p\rfloor=\lfloor x:(p\cdot q)\rfloor.
\end{equation}
To verify (\ref{nested divisions}) we use induction on $k\in \integers^*$. As $m=\lfloor m:2^0\rfloor$ the identity holds for $k=0$. 
\\Let us assume that the identity (\ref{nested divisions}) holds for $k=n$, that is: 
\[ \begin{array}{ccc} \lfloor\lfloor...\lfloor\lfloor m:\!\!\!\!\!&\underbrace{2\rfloor:2\rfloor...\rfloor :2}&\!\!\!\!\rfloor=\lfloor m:2^n\rfloor.\\ 
&    n                   & 
\end{array}
\]
In view of (\ref{nested division identity 2}) we have 
\[\begin{array}{ccc} \lfloor\!\!&\underbrace{\lfloor\lfloor...\lfloor\lfloor m:2\rfloor:2\rfloor...\rfloor :2\rfloor}&\!\! : 2\rfloor= \lfloor\lfloor m:2^n\rfloor:2\rfloor=\lfloor m:2^{n+1}\rfloor,\\ 
           &   \lfloor m:2^n\rfloor   & 
\end{array}\]
which completes the induction.
\end{proof}
\subsection{The main result}
The binary equivalent $[m]_2=r_1r_2...r_{n-1}\in{\mathcal{B}}_{n-1}$ of $m\in\I_{n-1}$ can be expressed, as described in (\ref{eq. of rs}) and (\ref{nested divisions}), in terms of elementary integer functions as follows: 
\begin{equation} \label{binary equvalent}
[m]_2\!=\!(\lfloor m\!:\!2^{n-2}\rfloor\!\!\!\!\!\mod\! 2, \lfloor m\!:\!2^{n-3}\rfloor\!\!\!\!\!\mod\! 2,...,\lfloor m\!:\!2^0\rfloor\!\!\!\!\!\mod\! 2).\end{equation}
The relation (\ref{binary equvalent}) induces the bijective transformation:
\begin{equation} \label{beta} 
\beta^{(n)}: \I_{n-1}\ni m \mapsto \beta^{(n)}(m)=[m]_2\in {\mathcal{B}}_{n-1}.\end{equation}
The composite $\tau^{(n)}\stackrel{{\rm def}}{=}\rho^{(n)}\circ \beta^{(n)}$ determines a bijection, which converts non-negative integers into arrays in $\mathfrak{R}_n$:
\begin{equation} \label{tau bijection}
\tau^{(n)}(m)\!=\!(\lfloor m\!:\!2^{n-2}\rfloor\!\!\!\!\!\mod\! 2, \lfloor m\!:\!2^{n-3}\rfloor\!\!\!\!\!\mod\! 2,...,\lfloor m\!:\!2^0\rfloor\!\!\!\!\!\mod\! 2,1).\end{equation}
Moreover, for every $n\in\naturals$ and every $m\in\naturals$ such that $m<2^{n-1}$:
\[[\tau^{(n)}(m)]_{10}=2m+1.\]
Finally, the composition of $\varphi^{(n)}$ (introduced in section \ref{sec:Representation of Non-trivials SEPs by 0s and 1s}) and $\tau^{(n)}$ yields the bijection $\chi^{(n)}\stackrel{{\rm def}}{=}\varphi^{(n)}\circ \tau^{(n)}$:
\[\begin{array}{rccl}
 &                        &  \tau^{(n)}                 &                       \\
&              \I_{n-1}  &   \longrightarrow     &  \mathfrak{R}_n \vspace{0.1in}\\
& \hspace{0.3in}\chi^{(n)}         &    \searrow           &  \downarrow  \ \ \ \varphi^{(n)}  \vspace{0.1in}\\
&                        &                  &{\mathcal{E}}_{n}    \end{array}\]
The composite  $\chi^{(n)}$ associates integers from $\I_{n-1}$ to complex numbers (the product values of the SEPs) and is defined once $\Hes_n$ is given.
The results of this section enable us to modify (\ref{preclosed form 1}) in order to reach a closed form (our initial quest) of $\det(\Hes_n)$ as will be described in the following Theorem.
\begin{theorem}\label{Closed form of Hessenbergian}
The closed form of $\det(\Hes_n)$ is:
\begin{equation}\label{closed form}
\det(\Hes_n)=\sum_{m=0}^{2^{n\!-\!1}-1}\chi^{(n)}(m).
\end{equation}
\end{theorem}
\begin{proof}
As $\tau^{(n)}:\I_{n-1}\ni m \mapsto \tau^{(n)}(m)\in \mathfrak{R}_n$ is bijective, every $r\in \mathfrak{R}_n$ can be replaced in (\ref{preclosed form 1}) by $\tau^{(n)}(m)$, $m\in \I_{n-1}$. Taking into account that $\chi^{(n)}$ is bijective, (\ref{preclosed form 1}) takes the form\vspace{-0.05in}
\[\det(\Hes_n)\!=\!\displaystyle\sum_{r\in \mathfrak{R}_n}\varphi^{(n)}(r) \!=\!\sum_{m\in\I_{n-1}} \varphi^{(n)}(\tau^{(n)}(m))
\!=\!\sum_{m=0}^{2^{n\!-\!1}-1} \chi^{(n)}(m),\]
as required.
\end{proof}
\subsection*{Examples}
\label{sec:Example}
To illustrate the closed form of $\det(\Hes_n)$ in (\ref{closed form}), we consider the Hessenbergians of order: $n=2,3,4$.
 
The expansion of $\det(\Hes_{2})$ consists of $2^1$ non-trivial SEPs, and $\I_1=\{0,1\}$. In view of (\ref{tau bijection}), the arrays $\tau^{(2)}(m)\in \mathfrak{R}_2$  are given by:
\[\begin{array}{lll}
\tau^{(2)}(0)&=&(\lfloor 0: 2^{2-2}\rfloor\!\!\!\!\mod 2,1)=(0\!\!\!\!\mod 2,1 )=(0,1)\\
\tau^{(2)}(1)&=&(\lfloor 1: 2^{2-2}\rfloor\!\!\!\!\mod 2,1)=(1\!\!\!\!\mod 2,1 )=(1,1).
       \end{array}
\]
Recalling that the non-standard factors are opposite-sign entries of the $\Hes_n$ superdiagonal, the non-trivial SEPs of $\Hes_2$ are:
\[\displaystyle \chi^{(2)}(0)= \varphi^{(2)}(\tau^{(2)}(0))=\varphi^{(2)}(0,1)=-h_{1,2}h_{2,1},\ \
\displaystyle \chi^{(2)}(1)=\varphi^{(2)}(\tau^{(2)}(1))=\varphi^{(2)}(1,1)=h_{1,1}h_{2,2}.
\]
Thus,\vspace{-0.1in} 
\[\det(\Hes_2)=\displaystyle\sum_{m=0}^1\chi^{(2)}(m)=h_{1,1}h_{2,2}-h_{1,2}h_{2,1}.\]

The expansion of $\det(\Hes_{3})$ consists of $2^2$ non-trivial SEPs, and $\I_2=\{0,1,2,3\}$. The arrays $\tau^{(3)}(m)\in \mathfrak{R}_3$  are given by:
\[\begin{array}{llccc}
\tau^{(3)}(0)\!\!\!\!&=(&\!\!\!\lfloor 0: 2^{3-2}\rfloor\!\!\!\!\mod 2,&\lfloor 0: 2^{3-3}\rfloor\!\!\!\!\mod 2,&1)\\
               &=(&\!\!\! 0\!\! \mod 2,&\ 0\!\! \mod 2,& 1)\\
               &=(&\!\!\!0,&0,&1)\vspace{0.1in}\\
\tau^{(3)}(1)\!\!\!\!&=(&\!\!\!\lfloor 1: 2^{3-2}\rfloor\!\!\!\!\mod 2,&\lfloor 1: 2^{3-3}\rfloor\!\!\!\!\mod 2,&1)\\
               &=(&\!\!\! 0\!\! \mod 2,&\ 1\!\! \mod 2,& 1)\\
               &=(&\!\!\!0,&1,&1)\vspace{0.1in}\\
\tau^{(3)}(2)\!\!\!\!&=(&\!\!\!\lfloor 2: 2^{3-2}\rfloor\!\!\!\!\mod 2,&\lfloor 2: 2^{3-3}\rfloor\!\!\!\!\mod 2,&1)\\
&=(&\!\!\! 1\!\! \mod 2,&\ 2\!\! \mod 2,& 1)\\
&=(&\!\!\!1,&0,&1) \end{array}\]
\[\begin{array}{llccc}
\tau^{(3)}(3)\!\!\!\!&=(&\!\!\!\lfloor 3: 2^{3-2}\rfloor\!\!\!\!\mod 2,&\lfloor 3: 2^{3-3}\rfloor\!\!\!\!\mod 2,&1)\\
&=(&\!\!\! 1\!\! \mod 
2,&\ 3\!\! \mod 2,&1)\\
&=(&\!\!\!1,&1,&1)
\end{array}\]
The non-trivial SEPs of $\Hes_{3}$ are listed below:
\[\begin{array}{l} 
\displaystyle \chi^{(3)}(0)=\varphi^{(3)}(\tau^{(3)}(0))=\varphi^{(3)}(0,0,1)=-h_{1,2}(-h_{2,3})h_{3,1}=h_{1,2}h_{2,3}h_{3,1}\vspace{0.05in}\\
\displaystyle \chi^{(3)}(1)=
\varphi^{(3)}(\tau^{(3)}(1))=\varphi^{(3)}(0,1,1)=-h_{1,2}h_{2,1}h_{3,3}\vspace{0.05in}\\
\displaystyle \chi^{(3)}(2)=
\varphi^{(3)}(\tau^{(3)}(2))=\varphi^{(3)}(1,0,1)=h_{1,1}(-h_{2,3})h_{3,2}=-h_{1,1}h_{2,3}h_{3,2}\vspace{0.05in}\\
\displaystyle \chi^{(3)}(3)=
\varphi^{(3)}(\tau^{(3)}(3))=\varphi^{(3)}(1,1,1)=h_{1,1}h_{2,2}h_{3,3}
\end{array}  \]
Thus,\vspace{-0.1in} 
\[\det(\Hes_3)=\displaystyle\sum_{m=0}^3\chi^{(3)}(m)=h_{1,2}h_{2,3}h_{3,1}-h_{1,2}h_{2,1}h_{3,3}-h_{1,1}h_{2,3}h_{3,2}+h_{1,1}h_{2,2}h_{3,3}.\]

Let us finally consider the $\det(\Hes_{4})$.
It consists of $2^3$ non-trivial SEPs, and $\I_3=\{0,1,...,7\}$. The arrays $\tau^{(4)}(m)\in \mathfrak{R}_4$ are given by:
\[\begin{array}{llcccc}
\tau^{(4)}(0)\!\!\!\!&=(&\!\!\!\lfloor 0: 2^{4-2}\rfloor\!\!\!\!\mod 2,&\lfloor 0: 2^{4-3}\rfloor\!\!\!\!\mod 2,&\lfloor 0: 2^{4-4}\rfloor\!\!\!\!\mod 2,&1)\\
               &=(&\!\!\! 0\!\! \mod 2,&\ 0\!\! \mod 2,&\ 0\!\! \mod 2,& 1)\\
               &=(&\!\!\!0,&0,&0,&1)\end{array}\]
               \[\begin{array}{llcccc}
\tau^{(4)}(1)\!\!\!\!&=(&\!\!\!\lfloor 1: 2^{4-2}\rfloor\!\!\!\!\mod 2,&\lfloor 1: 2^{4-3}\rfloor\!\!\!\!\mod 2,&\lfloor 1: 2^{4-4}\rfloor\!\!\!\!\mod 2,&1)\\
               &=(&\!\!\! 0\!\! \mod 2,&\ 0\!\! \mod 2,&\ 1\!\! \mod 2,& 1)\\
               &=(&\!\!\!0,&0,&1,&1)\end{array}\]
\[\begin{array}{llcccc}
\tau^{(4)}(2)\!\!\!\!&=(&\!\!\!\lfloor 2: 2^{4-2}\rfloor\!\!\!\!\mod 2,&\lfloor 2: 2^{4-3}\rfloor\!\!\!\!\mod 2,&\lfloor 2: 2^{4-4}\rfloor\!\!\!\!\mod 2,&1)\\
               &=(&\!\!\! 0\!\! \mod 2,&\ 1\!\! \mod 2,&\ 2\!\! \mod 2,& 1)\\
               &=(&\!\!\!0,&1,&0,&1)\vspace{0.1in}\\ 
\tau^{(4)}(3)\!\!\!\!&=(&\!\!\!\lfloor 3: 2^{4-2}\rfloor\!\!\!\!\mod 2,&\lfloor 3: 2^{4-3}\rfloor\!\!\!\!\mod 2,&\lfloor 3: 2^{4-4}\rfloor\!\!\!\!\mod 2,&1)\\
               &=(&\!\!\! 0\!\! \mod 2,&\ 1\!\! \mod 2,&\ 3\!\! \mod 2,&1)\\
               &=(&\!\!\!0,&1,&1,&1)\vspace{0.1in}\\
\tau^{(4)}(4)\!\!\!\!&=(&\!\!\!\lfloor 4: 2^{4-2}\rfloor\!\!\!\!\mod 2,&\lfloor 4: 2^{4-3}\rfloor\!\!\!\!\mod 2,&\lfloor 4: 2^{4-4}\rfloor\!\!\!\!\mod 2,&1)\\
               &=(&\!\!\! 1\!\! \mod 2,&\ 2\!\! \mod 2,&\ 4\!\! \mod 2,& 1)\\
               &=(&\!\!\!1,&0,&0,&1) \end{array}\]
\[\begin{array}{llcccc}
\tau^{(4)}(5)\!\!\!\!&=(&\!\!\!\lfloor 5: 2^{4-2}\rfloor\!\!\!\!\mod 2,&\lfloor 5: 2^{4-3}\rfloor\!\!\!\!\mod 2,&\lfloor 5: 2^{4-4}\rfloor\!\!\!\!\mod 2,&1)\\
               &=(&\!\!\! 1\!\! \mod 2,&\ 2\!\! \mod 2,&\ 5\!\! \mod 2,& 1)\\
               &=(&\!\!\!1,&0,&1,&1)\vspace{0.1in}\\
\tau^{(4)}(6)\!\!\!\!&=(&\!\!\!\lfloor 6: 2^{4-2}\rfloor\!\!\!\!\mod 2,&\lfloor 6: 2^{4-3}\rfloor\!\!\!\!\mod 2,&\lfloor 6: 2^{4-4}\rfloor\!\!\!\!\mod 2,&1) \\
               &=(&\!\!\! 1\!\! \mod 2,&\ 3\!\! \mod 2,&\ 6\!\! \mod 2,& 1)\\
               &=(&\!\!\!1,&1,&0,&1)\vspace{0.1in}\\
\tau^{(4)}(7)\!\!\!\!&=(&\!\!\!\lfloor 7: 2^{4-2}\rfloor\!\!\!\!\mod 2,&\lfloor 7: 2^{4-3}\rfloor\!\!\!\!\mod 2,&\lfloor 7: 2^{4-4}\rfloor\!\!\!\!\mod 2,&1)\\
               &=(&\!\!\! 1\!\! \mod 2,&\ 3\!\! \mod 2,&\ 7\!\! \mod 2,& 1)\\
               &=(&\!\!\!1,&1,&1,&1)   \end{array}\]
The non-trivial SEPs of $\Hes_{4}$ are listed below:
\[\begin{array}{l} 
\displaystyle \chi^{(4)}(0)=\varphi^{(4)}(\tau^{(4)}(0))=\varphi^{(4)}(0,0,0,1)=-h_{1,2}(-h_{2,3})(-h_{3,4})h_{4,1}=-h_{1,2}h_{2,3}h_{3,4}h_{4,1}\vspace{0.05in}\\
\displaystyle \chi^{(4)}(1)=
\varphi^{(4)}(\tau^{(4)}(1))=\varphi^{(4)}(0,0,1,1)=-h_{1,2}(-h_{2,3})h_{3,1}h_{4,4}=h_{1,2}h_{2,3}h_{3,1}h_{4,4}\vspace{0.05in}\\
\displaystyle \chi^{(4)}(2)=
\varphi^{(4)}(\tau^{(4)}(2))=\varphi^{(4)}(0,1,0,1)=-h_{1,2}h_{2,1}(-h_{3,4})h_{4,3}=h_{1,2}h_{2,1}h_{3,4}h_{4,3}\vspace{0.05in}\\
\displaystyle \chi^{(4)}(3)=\varphi^{(4)}(\tau^{(4)}(3))=\varphi^{(4)}(0,1,1,1)=-h_{1,2}h_{2,1}h_{3,3}h_{4,4}\vspace{0.05in}\\
\displaystyle \chi^{(4)}(4)=
\varphi^{(4)}(\tau^{(4)}(4))=\varphi^{(4)}(1,0,0,1)=h_{1,1}(-h_{2,3})(-h_{3,4})h_{4,2}=h_{1,1}h_{2,3}h_{3,4}h_{4,2}\vspace{0.05in}\\
\displaystyle \chi^{(4)}(5)=
\varphi^{(4)}(\tau^{(4)}(5))=\varphi^{(4)}(1,0,1,1)=h_{1,1}(-h_{2,3})h_{3,2}h_{4,4}=-h_{1,1}h_{2,3}h_{3,2}h_{4,4}\vspace{0.05in}\\
\displaystyle \chi^{(4)}(6)=
\varphi^{(4)}(\tau^{(4)}(6))=\varphi^{(4)}(1,1,0,1)=h_{1,1}h_{2,2}(-h_{3,4})h_{4,3}=-h_{1,1}h_{2,2}h_{3,4}h_{4,3}\vspace{0.05in}\\
\displaystyle \chi^{(4)}(7)=
\varphi^{(4)}(\tau^{(4)}(7))=\varphi^{(4)}(1,1,1,1)=h_{1,1}h_{2,2}h_{3,3}h_{4,4}.
\end{array}  \]
Thus,
\[\begin{array}{lll} \det(\Hes_4)\displaystyle =\sum_{m=0}^7 \chi^{(4)}(m)
\!\!\!&=&\!\!\! -h_{1,2}h_{2,3}h_{3,4}h_{4,1}+h_{1,2}h_{2,3}h_{3,1}h_{4,4}+h_{1,2}h_{2,1}h_{3,4}h_{4,3}-h_{1,2}h_{2,1}h_{3,3}h_{4,4}\vspace{0.05in}\\
&&\!\!\!+h_{1,1}h_{2,3}h_{3,4}h_{4,2} -h_{1,1}h_{2,3}h_{3,2}h_{4,4}- h_{1,1}h_{2,2}h_{3,4}h_{4,3}+h_{1,1}h_{2,2}h_{3,3}h_{4,4}.
\end{array}\]
The above results coincide with the determinant expansions derived from the Leibniz formula, yet excluding the trivial SEPs.
\section{The general solution of regular order LDEVCs} \label{sec:AscendingOrderDifferenceEquation}
We will finally describe the solution expressions in all three types of regular order LDEVCs (ascending-order, $ N $-order and unbounded-order), derived from the closed form of Hessenbergians.

As the solution matrices $\mbox{\boldmath$\rm \Xi$}  _n^{(i)}$, $\mbox{\boldmath$\rm P$}_n$, $\mbox{\boldmath$\rm G$}_n$ associated with the ascending order LDEVC (see section \ref{sec:SolutionsOfLinear DifferenceEquationsWithVariable CoefficientsInTermsOfHessenbergians}) are all in lower Hessenberg form, the determinants of the fundamental, particular and general solution matrices are all Hessenbergians. As a consequence, the formula (\ref{closed form}) is directly applicable to each solution determinant.

More specifically, the general solution of the ascending order LDEVC, is obtained by identifying the general solution matrix 
$\mbox{\boldmath$\rm G$}_n$ in (\ref{GSMat}) with the matrix:
\begin{equation} \label{Hessenberg Matrix3}
 \Hes_{n+1}=\left(\begin{array}{cccccc}
 h_{1,1}      &  h_{1,2}   &   0        & ... &   0           & 0  \\
 h_{2,1}      &  h_{2,2}   &   h_{2,3}  & ... &   0           & 0  \vspace{-0.05in}\\
 .          &  .        &   .        & ... &   .           & .  \vspace{-0.1in}\\
  .          &  .        &   .        & ... &   .           & .  \vspace{-0.1in}\\
      .          &  .        &   .        & ... &   .           & .\vspace{-0.05in}\\
h_{n, 1} & h_{n,2}& h_{n,3} & ... & h_{n, n} & h_{n, n+1} \\
  h_{n+1,1}      &  h_{n+1,2}   & h_{n+1,3}     & ... & h_{n+1, n} & h_{n+1, n+1}
\end{array}\right).
\end{equation}
In other words, the entries $h_{i,j}$ of $\Hes_{n+1}$ are assigned with the values: 
\[h_{i,j}=\left\{\begin{array}{cl}\displaystyle g_{i-1}\!-\!\sum_{k=0}^{N-1}a_{i-1,k}\ y_{k-N} & {\rm if} \ 1\le i\le n+1\ {\rm and}\ j=1, \\
a_{i-1,N+j-2} &  {\rm if} \ 1\le i\le n+1\ {\rm and}\ j=2,...,i+1, \vspace{0.1in}\\                      0          &  {\rm otherwise}.
\end{array}\right.\]
Applying the closed form (\ref{closed form}) to $\Hes_{n+1}$, the general solution of the ascending order LDEVC in (\ref{n-term general solution}) takes the following (closed) form:
\begin{equation} \label{closed form of general solution}
y_n=(-1)^n\frac{\displaystyle\sum_{m=0}^{2^{n}-1}\chi^{(n+1)}(m)}{\displaystyle\prod_{i=0}^{n}a_{i,i+N}}.
\end{equation}
We can further reduce (\ref{closed form of general solution}) by changing the 1st column entries of the general solution matrix (\ref{Hessenberg Matrix3}). In particular, we assign $\Hes_{n+1}=(h_{i,j})_{1\le i,j\le n+1}$ (or $\mbox{\boldmath$\rm G$}_n$) with the entries: 
\begin{equation}\label{new entries} h_{i,j}=\left\{\begin{array}{cl}\displaystyle (-1)^{n}\frac{\displaystyle g_{i-1}\!-\!\sum_{k=0}^{N-1}a_{i-1,k}\ y_{k-N}}{\displaystyle\prod_{i=0}^{n}a_{i,i+N}} & {\rm if} \ 1\le i\le n+1\ {\rm and}\ j=1, \vspace{0.1in}\\
a_{i-1,N+j-2} &  {\rm if} \ 1\le i\le n+1\ {\rm and}\ j=2,...,i+1, \vspace{0.1in}\\                      0          &  {\rm otherwise}
\end{array}\right.\end{equation}
Applying the multilinear property of determinants with respect to the 1st column of $\Hes_{n+1}$,  as defined in (\ref{new entries}), the general solution described in (\ref{n-term general solution}) can be expressed as a single Hessenbergian: 
\begin{equation} \label{General Solution Reduced}
y_n=\det(\Hes_{n+1}).
\end{equation}
Finally, the expression (\ref{closed form}), applied to (\ref{General Solution Reduced}), gives the  ascending order LDEVC general solution a more condensed, alternative to (\ref{closed form of general solution}), closed form:
\begin{equation} \label{closed form of general solution2}
y_n=\displaystyle\sum_{m=0}^{2^{n}-1}\chi^{(n+1)}(m).
\end{equation}

The $N$th order LDEVC is also associated with the solution matrices (fundamental, particular and general) $\mbox{\boldmath$\rm \Xi$}  _n^{(i)}$, $\mbox{\boldmath$\rm P$}_n$ and $\mbox{\boldmath$\rm G$}_n$, respectively. Its general solution is represented by the formulas (\ref{closed form of general solution}) (or (\ref{closed form of general solution2})). In this case, however, the solution matrices are even more sparse. In particular, the fundamental solution matrix $\mbox{\boldmath$\rm \Xi$}  _n^{(i)}$ is a band matrix in which additional zero entries are grouped at its bottom left corner. This produces additional zero, but in our terminology non-trivial, SEPs that are also included in the formula.

In the unbounded order case (where $N=0$), homogeneous solutions do not exist. Thus $\mbox{\boldmath$\rm G$}_n=\mbox{\boldmath$\rm P$}_n$ and the unique solution of the unbounded order LDEVC is derived from the closed form expression of the Hessenbergian $\det(\mbox{\boldmath$\rm P$}_n)$ in (\ref{n-term particular solution}).

\end{document}